\documentclass[a4paper,review]{article}
\usepackage{amssymb,amsfonts,amsmath,amsthm} %%,, 

\newcommand{\beq}{\begin{equation}}
\newcommand{\beqnt}{\begin{equation}\nonumber}
\newcommand{\eeq}{\end{equation}}
\newcommand{\fref}[1]{{\rm(\ref{#1})}}                      %%%%%%%%%% formula ref      %%%%%
\newcommand{\PC}{{\cal P}}%%%%%%%%%   допустимые значения управления
\newcommand{\QC}{{\cal Q}}%%%%%%%%%   допустимые значения помехи
\newcommand{\FC}{{\cal F}}%%%%%%%%%   правая часть дифф включения
\newcommand{\UA}{{\mathbb U}}%%%%%%%%%  стратегия с полной памятью
\newcommand{\UB}{{\mathbf U}}%%%%%%% всякие множества управлений
\newcommand{\VB}{{\mathbf V}}%%%%%%% всякие множества помехи
\newcommand{\Upr}{{\cal U}}%%%%%%    множетво программных управлений - реализаций управлений
\newcommand{\Vpr}{{\cal V}}%%%%%%    множетво программных помех - реализаций помех
\newcommand{\Uqs}{{\mathbf Q}}%%%%%%    множетво сосредоточенных однозначных квазистратегий управления
\newcommand{\Uxu}{{\mathbf S}}%%%%%%    множетво стратегий управления с полной пам. о движении и управлении
\newcommand{\Part}{\Delta}%%    разбиение
\newcommand{\Partk}{{\Part_k}}%%   k-е разбиение
\newcommand{\Pparset}{{\Part_T}}%%    множество всех разбиений при всех начальных моментах
\newcommand{\nPart}{n_{\Part}}%%   мощность разбиения
\newcommand{\nPartk}{n_{\Partk}}%%   мощность k-го разбиения
\newcommand{\cPart}{(\tau_i)_{i\in\nint{0}{\nPart}}}%%   кортеж разбиения
\newcommand{\cPartk}{(\tau_{ki})_{i\in\nint{0}{\nPart}}}%%   кортеж k-го разбиения
\newcommand{\epar}{\varepsilon}%%%%%%  индекс в eps - оптимальной стратегии
\newcommand{\Ue}{{\UA_\epar}}%%%%%%  обозначение  eps - оптимальной стратегии
\newcommand{\Uepar}{{\UB^{\Part}_\epar}}       %%%%%%  обратная связь на разбиении для  eps-стратегии
\newcommand{\Uepark}{{\UB^{\Partk}_\epar}}       %%%%%%  обратная связь на k-ом разбиении для  eps-стратегии
\newcommand{\Uep}[1]{{\UB^{\Part}_{\epar#1}}}   %%%%%%  #1-й элемент обратной eps-стратегии 
\newcommand{\Uepk}[1]{{\UB^{\Partk}_{\epar#1}}} %%%%%%  #1-й элемент обратной связи eps-стратегии для k-го разбиения
\newcommand{\RA}{{\mathbb R}}%%%%%%%%%  действительные числа
\newcommand{\NA}{{\mathbb N}}%%%%%%%%%  натуральные числа
\newcommand{\mfv}{\mu_v}%%%%%%%%%  модуль непрерывности правой части по 4 аргументу
\newcommand{\mfu}{\mu_u}%%%%%%%%%  модуль непрерывности правой части по 4 аргументу
\newcommand{\mes}{\lambda}%%%%%%%%%  мера множества

\newcommand{\Xc}{{\cal X_\textsc{c}}}                %%%%%%%%%%% пучек констр. движение при компактных множествах помех
\newcommand{\Xs}{{\cal X}}                           %%%%%%%%%%% пучек констр. движение при произв. помехах
\newcommand{\Xpr}{{\cal X_\textsc{p}}}               %%%%%%%%%%% пучек констр. движение при программных помехах
\newcommand{\Gs}{{\Gamma}}                           %%%%%%%%%%% оптимальная гарантия при произв. помехах
\newcommand{\Gc}{{\Gamma_\textsc{c}}}                         %%%%%%%%%%% оптимальная гарантия при L2 комп. помехах
\newcommand{\Gp}{{\Gamma_\textsc{p}}}                         %%%%%%%%%%% оптимальная гарантия при прогр. помехах
\newcommand{\Gq}{{\Gamma_\textsc{q}}}                         %%%%%%%%%%% оптимальная в классе квазистрат.
\newcommand{\Wq}{{\cal W}}                         %%%%%%%%%%% движения оптимальных квазистрат.
\newcommand{\qndx}{\gamma}                          %%%%%%%%%%  показатель качества
\newcommand\argmin{\operatornamewithlimits{\mathrm{argmin}}}

\newcommand\diam[1]{\ensuremath{\operatorname{{\mathrm D}(#1)}}}
\newcommand\diamin[1]{\operatorname{{\mathrm d}(#1)}}
\newcommand*\dist[1]{\operatorname{\mathbf{d}^H_{#1}}}%%%%%%%%%% полу-расстояние Хаусдорфа от второго до первого множествами
\newcommand\close{\operatorname{\mathbf cl}}

\newcommand\comp[2]{{\mathbf{comp}_{#2}{(#1)}}}
\newcommand\cVpr{\comp{\Vpr}{L_2}}
\newcommand{\SP}[2]{{\left\langle{#1},{#2}\right\rangle}}
\newcommand{\nint}[2]{{{#1}..{#2}}}
\newcommand{\mydef}{=}%% пред.вар.:{\operatorname{\triangleq}}%%%%%%%%%%%% знак <<равно по определению>>
\newcommand\mysim[1]{\mathrel{\mathop{\sim}\limits_{#1}}}%%%%%%%%%%%%%% знак отношения экв. параметризованный
\newcommand{\myinf}{\operatornamewithlimits{\inf\vphantom{\sup}}}%%%%%%%%%%%%%%%%%%%%% знак inf выравненый <<под sup>>
\newcommand{\mymin}{\operatornamewithlimits{\min\vphantom{\sup}}}%%%%%%%%%%%%%%%%%%%%% знак min выравненый <<под sup>>
\newcommand{\mymax}{\operatornamewithlimits{\max\vphantom{\sup}}}%%%%%%%%%%%%%%%%%%%%% знак max выравненый <<под sup>>
\newcommand{\mylim}{\operatornamewithlimits{\lim\vphantom{\sup}}}%%%%%%%%%%%%%%%%%%%%% знак lim выравненый <<под sup>>
\newcommand\aee{a.a.\ }							   %%%%%%%%%%%%%%%%% almost all

\newtheorem{theorem}{Theorem}
\newtheorem{lemma}{Lemma}
\theoremstyle{remark}
\newtheorem{remark}{Remark}
\newtheorem{example}{Example}

\title{On non-improvability of full--memory strategies in problems of optimization of the guaranteed result}

\author{Dmitrii Serkov \thanks{Krasovskii Institute of Mathematics and Mechanics Ural Branch of Russian Academy of Sciences; S.Kovalevskoy, 16, Yekaterinburg, 620990, Russia\endgraf
Yeltsyn Ural Federal University; Mira, 19, Yekaterinburg, 620003, Russia.}\\ e-mail: serkov@imm.uran.ru}

\begin{document}

\maketitle

\begin{abstract}
The paper addresses the problem of optimization of a guaranteed (worst case) result for a control system driven by a controlling side in presence of a dynamical disturbance. The disturbances as functions of time are subject to functional constraints belonging to a given family of constraints. The latter family is known to the controlling side that does not observe the disturbance and uses full-memory strategies to form the control actions. The study is focused on the case where disturbance varies in open-loop disturbances chosen in advance and the case where the disturbances are restricted to a $L_2$--compact set fixed in advance but unknown to the controlling side. In these cases it is shown that the optimal guaranteed result is non-improvable in the sense that it coincides with that obtained in the class of quasi-strategies -- nonantisipatory transformations of disturbances into controls. An $\varepsilon$--optimal full-memory strategy is constructed. An illustrative nonlinear example is given.
\end{abstract}



\section{Introduction}

This work is related to the theory of guaranteeing positional control (see \cite{KraSub88e,SubChe81e}). The theory focuses on assessment of the optimal guaranteed result --- the minimax of a cost functional --- for the controlling side that is faced with the disturbance in the process of steering a dynamical control system. Here, we study properties of the optimal guaranteed result and describe an optimal strategy for the controlling side in the case where the dynamical disturbance, non-observable by the control, is subject to a functional constraint belonging to a given family of constraints.
In particular, the control problem under the action of a disturbance, that is known to be independent both of the system's state and of control actions can be considered as a problem with  functional constraint on the disturbances. The examples of such situation are numerous: a physical object under the influence of natural forces (the wind during the aircraft landing); the mass behavior relative to an individual (the economic forces relative to some enterprise), etc.

Control problems with additional functional constraints imposed on the input dynamical disturbances have been studied in various formalizations. The simplest functional constraint restricts the disturbances to the open-loop ones. In \cite{KrasIZVD70e,Kra71DANe,KrasUDS85e} the maximin open-loop constructions (including the stochastic ones) use open-loop disturbances to find the optimal guaranteed result and optimal closed-loop strategies in control problems with non-constrained disturbances. In \cite{BarSub70e,BarSub71e} properties of linear control systems in the cases of open-loop disturbances, disturbances generated by continuous feedbacks, and disturbances formed by upper semicontinuous set-valued closed-loop strategies were compared. In \cite{Kry91} assuming that the disturbances are restricted to an unknown $L_2$-compact set, it was shown that the optimal guaranteed result achieved in the class of the full-memory closed-loop control strategies equals that achieved by the 'fully informed controller' allowed to control the system using quasi-strategies --- nonantisipatory open-loop control responses to disturbance realizations \cite{SubChe81e}; in this sense the full-memory closed-loop strategies are uninprovable. In \cite{Ser_DAN2013et} considering the problem setting proposed in \cite{Kry91} in the case of a continuous cost functional, new unimprovability conditions for full-memory control strategies were given and an optimal full-memory control strategy allowing numerical implementation was constructed.

In this work, for the case of a continuous cost functional we show that, firstly, the guaranteed control problem with open-loop disturbances is equivalent to that with the disturbances restricted to $L_2$-compact sets, and, secondly, the optimal guaranteed results achieved by the controlling side in the class of full--memory control strategies under these two types of constraints on the disturbances are equal to that achieved in the class of quasi--strategies. In showing the results, we use elements of theory of robust dynamical inversion of control systems (\cite{KryOsi83e,KryOsi95}). The idea is that on small time intervals we replace the useful control by a series of test control actions.  By observation the system's responses to these test control actions, we reconstruct a disturbance that is approximately equivalent to the actual one, which allows the controlling side to operate with the efficiency of quasi-strategies.

This article is the English version of the Russian publication \cite{Ser_TRIMM2014e}.

\section{Definitions}

Consider a control system
\begin{equation}
\label{sys}
\begin{cases}
\dot x(\tau)=f(\tau,x(\tau),u(\tau),v(\tau)),&\tau\in T\mydef[t_0,\vartheta]\subset\RA,\\
x(t_0)=z_0\in G_0\subset\RA^n,
\end{cases}
\end{equation}
$$%%\label{rest_u_v}
u(\tau)\in\PC\subset\RA^p,\ v(\tau)\in\QC\subset\RA^q,\quad\tau\in T.
$$
%Here symbol $\mydef$ means ``equal by definition'';
Here
$\PC$, $\QC$, and $G_0$ are compact sets;
and
$f(\cdot): T \times \RA^n \times \PC \times \QC \mapsto \RA^n $ is continuous,
locally Lipschitz in the second argument and such that for some $ K\ge0$ the inequality
$$
\sup_{(\tau,u,v)\in T\times\PC\times\QC}\|f(\tau,x,u,v)\|\le K(1+\|x\|).
$$
holds for all $x\in\RA^n$ ($\|\cdot\|$ denotes the norm in an Euclidian space).
{\em Controls} $u(\cdot): T \mapsto \PC $ and {\em disturbances} $v(\cdot): T \mapsto \QC $ are supposed to be  Lebesgue measurable. Denote by $\Upr$ the set of all controls and by $\Vpr$ the set of all disturbances.

For arbitrary $(t_*,x_*)\in T\times\RA^n$, $u(\cdot)\in\Upr$, $v(\cdot)\in\Vpr$ we denote
by $x(\cdot,t_*,z_*,u(\cdot),v(\cdot))$
the (unique) Carath\'eodory solution of (\ref{sys}) (see \cite[II.4]{Wargae}) defined on $[t_*,\vartheta]$ and satisfying the initial condition $x(t_*)=x_*$.
We fix a compact set $G\subset T\times\RA^n$
such that
$ (t,x(t,t_0,z_0,u(\cdot),v(\cdot))) \in G $
for all
$ t \in T $,
$ z_0 \in G_0 $
$u(\cdot)\in\Upr$,
$v(\cdot)\in\Vpr$,
and denote
\beq\label{kappa_G}
\varkappa\mydef\max_{\scriptsize(\tau,x)\in G\atop u\in\PC, v\in\QC}\|f(\tau,x,u,v)\|.
\eeq
%%%%%%

A set $\Part\mydef(\tau_i)_{i\in\nint{0}{\nPart}}$
where
$\tau_0=t_0$, $\tau_{i-1}<\tau_{i}$, $\tau_{\nPart}=\vartheta$
will be called a {\em partition} (of
interval $T$).
Denote by $\Pparset$ the set of all partitions.
For a partition $\Part\mydef(\tau_i)_{i\in\nint{0}{\nPart}}$ and
a $t\in T$ set
$$
i_t\mydef\max_{i\in\nint{0}{\nPart}\atop\tau_i\le t}i,\quad \diamin{\Part}\mydef\min_{i\in\nint{1}{(\nPart-1)}} \tau_i-\tau_{i-1},\quad \diam{\Part}\mydef\max_{i\in\nint{1}{\nPart}}\tau_i-\tau_{i-1}.
$$

Following \cite{Kry91}, define full--memory control strategies used by the controlling side.
For every
$ \tau_*, \tau^* \in T $
where
$ \tau^* > \tau_* $
denote
by
$ {\cal U}|_{[\tau_*, \tau^*)} $
the set of the restrictions of
all controls to
$ [\tau_*, \tau^*) $.
Given a partition  $\Part\mydef(\tau_i)_{i\in\nint{0}{\nPart}}$,
any family
$\UB^\Part\mydef(\UB^\Part_i(\cdot))_{i\in\nint{0}{(\nPart-1)}}$
where
$
\UB^\Part_i(\cdot): C([t_0, \tau_i], \RA^n) \mapsto {\cal U}|_{[\tau_i, \tau_{i+1})} $
$ (i\in\nint{0}{(\nPart-1)}) $
will be called a
\emph{full-memory feedback for partition $\Part$}.
Every family
$ \UA = (\UB^{\Part})_{\Part \in \Pparset} $
where
$ \UB^{\Part} $
is a
full-memory feedback for $\Part$
will be called a
{\em full-memory control strategy}.
We denote by
$ \Uxu $
the set of all full-memory control strategies.

Given
a $ z_0 \in G_0 $,
a partition
$\Part\mydef(\tau_i)_{i\in\nint{0}{\nPart}}$,
a full-memory feedback
$ \UB^\Part = (\UB^\Part_i (\cdot))_{i \in \nint{0}{(n_{\Part}-1)}} $
for
$\Part $
and
a disturbance $ v(\cdot) \in \Vpr $,
the
function
$ x(\cdot) = x(\cdot, t_0, z_0, u(\cdot), v(\cdot)) $
where
$ u(\cdot) \in \Upr $
is such that
$ u(t) = $
$ \UB^\Part_i(x(\cdot)|_{[t_0,\tau_i]})(t) $
for all
$ t \in [\tau_i, \tau_{i+1}) $
and all
$ i\in\nint{0}{(\nPart-1)} $
will be called
the (system's) {\em motion} originating at $ z_0 $ and corresponding to $\Part$,  $ \UB^\Part $ and $ v(\cdot) $;
we denote $ x(\cdot) $ and $ u(\cdot) $ by $ x(\cdot,z_{0},\UB^{\Part},v(\cdot)) $ and  $ u(\cdot,z_{0},\UB^{\Part},v(\cdot)) $, respectively.

For every
$ z_0 \in G_0 $,
every full-memory control strategy
$ \UA = (\UB^{\Part})_{\Part \in \Pparset} \in \Uxu $
and every nonempty set
of disturbances,
$\VB \subseteq \Vpr $,
we define the
{\em bundle of
motions}
originating from
$ z_0 $
and corresponding to
$ \UA $
and
$\VB $
to be the set 
$ X(z_0,\UA,\VB) $
of all
$ x (\cdot) \in C (T; \RA^n) $
with the following property:
there is a sequence
$ \{(z_ {0k}, v_k (\cdot), \Partk, \UB^{\Partk}) \}_{k=1}^{\infty} $
in
$ G_0 \times \VB \times \Pparset \times \UA $
such that
$\lim_{k \to\infty} z_{0k} = z_0$, $\lim_{k \to\infty} \diam{\Partk} = 0$
and
$ x(\cdot,z_{0k},\UB^{\Partk},v_k(\cdot)) \rightarrow x(\cdot) $
in
$ C(T;\RA^n) $.

In the above definition, $\VB$ is a functional constraint on the disturbances. Generally, we assume that the controlling side does not know $ \VB $ but knows a class of functional constraints $ \VB $ belongs to. The latter class gives the controlling side general information on the disturbance behavior but does not provide information on the exact one.

We will consider three cases of functional constraints on the disturbances. One case is the lack of constraints; in this exceptional case $ \VB = \Vpr $ is known to the controlling side. Another case assumes that the disturbance can vary within some $L_2$--compact set; in this case the controlling side considers the class of all $L_2$-compact $\VB\subset\Vpr $ as potential candidates for constraining the disturbances. The final case assumes that the disturbance is allowed to be open-loop only; in this case all one--element subsets of $\Vpr$ represent the potential constraints on the disturbances.

In the above cases,
for every
$ z_0 \in G_0 $
and every
full-memory control strategy
$ \UA \in \Uxu$,
we define the bundles of the
system's
{\em motions}
originating at
$ z_0 $
under
$ \UA $
subject to
{\em arbitrary disturbances},
{\em $ L_2 $-compactly constrained disturbances},
and
{\em open-loop disturbances}
as, respectively,
\begin{eqnarray*}
\Xs(z_0,\UA)&\mydef&X(z_0,\UA,\Vpr),\notag\\%\label{def_X*}
\Xc(z_0,\UA)&\mydef&\bigcup_{\VB\in\cVpr}X(z_0,\UA,\VB),\\%\label{def_Xc}
\Xpr(z_0,\UA)&\mydef&\bigcup_{v(\cdot)\in\Vpr}X(z_0,\UA,\{v(\cdot)\}) ;\notag%\label{def_Xpr}
\end{eqnarray*}
here $\cVpr$ denotes the family of all $L_2(T; \RA^q)$-compact subsets of $ \Vpr $.

\begin{remark}
The definition of $\Xs(z_0, \UA)$ is a straightforward generalization of the definition of the
set of constructive motions
generated by a closed-loop control strategy (see \cite{KraSub88e}).
The definition of $\Xc(z_0, \UA)$ follows \cite{Kry91}.
\end{remark}

According to the definitions, $\Xpr(z_0, \UA) \subseteq\Xc(z_0, \UA) \subseteq \Xs (z_0, \UA)$ holds for all $z_0\in G_0$ and $\UA\in\Uxu$.

\begin{remark}
In \cite {Ser_UDSU_09e} it was shown that generally $\Xpr(z_0, \UA)\neq\Xs (z_0, \UA)$.
A similar reasoning can lead to a statement that generally
$\Xc(z_0, \UA)\neq\Xs (z_0, \UA)$.
\end{remark}

Let the controlling side evaluate the quality of the system's motions by a continuous {\em cost functional} $ \qndx(\cdot): C (T; \RA^n)\mapsto\RA $. The controlling side seeks then to chose a full-memory control strategy that guarantees the minimum value for the supremum of $ \qndx(x(\cdot)) $ over the system's motions $ x(\cdot) $ corresponding to the chosen control strategy and all disturbances that are allowed within the given constraints.

Following \cite {KraSub88e}, and \cite {SubChe81e}, we call
$$
\Gs (z_0, \UA) \mydef \sup_ {x (\cdot) \in \Xs (z_0, \UA)} \qndx (x(\cdot))
$$
the {\em guaranteed result} at $z_0 \in G_0$ for a full-memory control strategy $\UA$ against \emph{arbitrary disturbances}; and we call
$$
\Gs(z_0) \mydef \inf_{\UA \in \Uxu} \Gs (z_0, \UA)
$$
the {\em optimal guaranteed result} at $z_0 \in G_0$ in the class of the full-memory control strategies $\Uxu$, against \emph{arbitrary disturbances}.
Similarly, we call
$$
\Gc (z_0, \UA) \mydef \sup_ {x (\cdot) \in \Xc (z_0, \UA)} \qndx (x(\cdot))
$$
the {\em guaranteed result} at $z_0 \in G_0$ for a full-memory control strategy $\UA $ against $L_2$--\emph{compactly constrained disturbances} and we call
$$
\Gc(z_0) \mydef \inf_{\UA \in \Uxu} \Gc(z_0, \UA)
$$
the {\em optimal guaranteed result} at $ z_0 \in G_0 $ in $\Uxu$ against $L_2$-\emph{compactly constrained disturbances}. 

Finally, we call
$$
\Gp(z_0, \UA) \mydef \sup_ {x (\cdot) \in \Xpr(z_0, \UA)} \qndx (x(\cdot)).
$$
the {\em guaranteed result} at $ z_0 \in G_0 $ for a full-memory control strategy $\UA  $ against \emph{open-loop disturbances}; and we call
$$
\Gp(z_0) \mydef \inf_{\UA \in \Uxu} \Gp(z_0, \UA).
$$
the  {\em optimal guaranteed result} at $ z_0 \in G_0 $ in $\Uxu$ against \emph{open-loop disturbances}.

Along with the full-memory control strategies, we introduce, after \cite{SubChe81e}, control quasi-strategies --- nonantisipatory transformations of disturbances into controls. The controlling side uses quasi--strategies if he/she is fully informed about the current histories and current values of the disturbance. A {\em control quasi-strategy} is a mapping $\alpha(\cdot): {\Vpr}\mapsto {\Upr}$ satisfying the following {\em non-anticipativity condition}: $\alpha(v(\cdot))|_{[t_0,\tau]}= \alpha (v '(\cdot))|_{[t_0,\tau]}$ for any $ \tau\in T $, $v (\cdot),v '(\cdot) \in \Vpr$ such that $v(\cdot)|_{[t_0,\tau]} = v'(\cdot)|_{[t_0,\tau]}$. We denote by $\Uqs$ the set of all control quasi-strategies. For every $ z_0 \in G_0 $ and every control quasi-strategy $ \alpha(\cdot)$,  we call
$$
\Xs(z_0,\alpha(\cdot))\mydef\{x (\cdot, t_0, z_0,\alpha(v(\cdot)),v(\cdot))\mid v(\cdot)\in{\Vpr}\}
$$
the bundle of {\em motions} originating at $ z_0 $ under $\alpha(\cdot)$. For every $ z_0 \in G_0 $ the value 
$$
\Gq(z_0, \alpha(\cdot))=\sup_{x(\cdot)\in\Xs(z_0,\alpha(\cdot))}\qndx(x(\cdot))
$$
is called the \emph{guaranteed result} at $z_0 $  for a control-quasi-strategy $ \alpha(\cdot) $ against {\em arbitrary disturbances}, and
$$
\Gq(z_0) \mydef\myinf_{\alpha(\cdot)\in\Uqs}\Gq(z_0, \alpha(\cdot)))
$$
is called the \emph{optimal guaranteed result} at $z_0 $ \emph{in the class of the control quasi-strategies}, $\Uqs$, against {\em arbitrary disturbances}.

\begin{remark}
Similar to how it was done above, one can also define the optimal guaranteed results at $ z_0 \in G_0 $ in $\Uqs$ against $L_2$--compactly constrained disturbances and against open-loop disturbances. However, these definitions would lead to the same value; the control quasi-strategies are insensitive to the $ L_2 $-compact and open-loop constraints on the disturbances.
\end{remark}

The next statement is obvious.

\begin{theorem}% \ label {theorem-gam}
For every $z_0\in G_0 $
\beq\label{neq_gam}
\Gq(z_0) \leq \Gp (z_0) \leq \Gc(z_0) \leq \Gs (z_0).
\eeq
\end{theorem}

\begin{remark}
As follows from  \cite{KraSub88e}, and \cite{SubChe81e}, for every $ z_0\in G_0$ all the inequalities in (\ref{neq_gam})  turn into equalities if
\begin{equation}
\label{cond_SPSG}
\min_{u\in\PC}\max_{v\in\QC}\SP{s}{f(\tau,x,u,v)}=\max_{v\in\QC}\min_{u\in\PC}\SP{s}{f(\tau,x,u,v)}
\end{equation}
for all $(\tau,x)\in G$, $s\in\RA^n$.
In that case neither the $ L_2 $-compact, nor open-loop constraints on the disturbances change the
optimal guaranteed result.
\end{remark}

In this paper we do not assume (\ref{cond_SPSG})  to be satisfied 
for all $(\tau,x)\in G$, $s\in\RA^n$.
In such circumstances, some inequalities given in (\ref{neq_gam}) can be strict. 
Examples of the situations where the first and  last elements in the chain (\ref{neq_gam})
differ are well known (see \cite[Chapter VI, \S1]{SubChe81e}).
For the case where the cost functional $\qndx$  is uniformly $(L^1,\delta)$-continuous on the set of all motions of system (\ref{sys}) but is not continuous on $C(T, \RA^n)$, an example of the situation where the last inequality in (\ref{neq_gam}) is strict, was constructed in \cite{Kry91}
(where one can also find a definition of the uniform $(L^1,\delta)$-continuity).
For $\qndx$  continuous on $C(T, \RA^n)$ a similar example was given in \cite{Ser_UDSU_10e}.

Among the optimal guaranteed results 
(at a
$ z_0 \in G_0 $)
given in
 \fref{neq_gam} 
the smallest one is the optimal guaranteed result in the class of the control quasi-strategies. 
We address a question whether
the optimal guaranteed 
result 
(at 
$ z_0 $)
in the class of the full-memory control strategies
against either
open-loop disturbances,
or 
 $ L_2 $-compactly constrained disturbances
 coincides with that in the class of quasi-strategies.
 If the answer is positive,
 the class of the full-memory control strategies, $\Uxu$, is 
 {\em non-improvable} against 
 a corresponding type of functional constraints on the disturbances.
In that situation, the use of any information on the past and current values of the
 actual disturbance does not allow the controlling side to improve the value of the optimal guaranteed result at any 
$ z_0 \in G_0 $, provided  the disturbance subject to
the corresponding type of functional constraints.

In \cite{Kry91} it was shown that in the case of a uniformly $(L^1, \delta)$-continuous cost functional the one-to-one correspondence in the mapping $ v\mapsto f(t, x, u, v)$ for all $ (t,x,u) \in T \times \RA^n \times \PC $ is sufficient for the non-improvability of $\Uxu$ against $L_2$-compactly constrained disturbances. In \cite{Ser_DAN2013et} for the case of a cost functional continuous in $ C (T, \RA^n) $ a less restrictive sufficient condition for the non-improvability of $\Uxu$  against the $L_2$-compactly constrained disturbances and the corresponding optimal full-memory control strategy was constructed. Without that additional condition to be assumed, a question if the first or/and second inequality given in \fref {neq_gam} turns into an equality for a cost functional continuous in $ C (T, \RA^n) $ has so far remained open.
In this paper we give a positive answer to the question.
Namely, we show that
at every
$ z_0 \in G_0 $
the class of the full-memory control strategies, 
$\Uxu$,
is non-improvable
against both
$L_2$-compactly constrained and open-loop disturbances.

%%%%%%%%%%%%%%%%%%%%%%%%%%%%%%%%%%%%%%%%%
\section {Non-improvability of full-memory control strategies against
$L_2$-compactly constrained and open-loop disturbances}
%%%%%%%%%%%%%%%%%%%%%%%%%%%%%%%%%%%%%%%%%

In this section we construct a family $(\Ue)_ {\epar>0} $
of full-memory control strategies,
$\Ue = (\Uepar)_{\Delta \in \Delta_T} $
$ (\varepsilon > 0) $,
such that 
for a given
$ z_0 \in G_0 $
$$
\limsup_{\varepsilon \rightarrow 0} \Gc(z_0, \Ue) \leq \Gq(z_0). 
$$
Then, in view of
(\ref{neq_gam}),
we get
$$
\Gc(z_0) = \Gp(z_0) = \Gq(z_0), 
$$
which implies that the 
full-memory control strategies 
are non-improvable
at
$ z_0 $
against both
$L_2$-compactly constrained and open-loop disturbances.

The process of operation of the full-memory 
feedback
$\Uepar= (\Uep{i}(\cdot))_{i\in\nint{0}{(\nPart-1)}} $
for a partition
$\Part=\cPart$ includes on-line simulation of a motion $ y(\cdot) $
of an auxiliary copy of system (\ref{sys}),
which we call the $y$-{\em model},
on every interval
$ [\tau_i, \tau_{i+1}) $.
In the simulation process,
the controlling side implements the robust dynamical inversion approach (\cite{KryOsi83e,KryOsi95}). He/she identifies a 'surrogate' disturbance $ \bar{v}_i $ that mimics the affect of the actual disturbance on the system, and lets the 'surrogate' disturbance operate in the $y$-model. To identify the 'surrogate' disturbance $ \bar{v}_i $, in a small final part of the time interval $ [\tau_{i-1}, \tau_{i}) $ the controlling side implements a series of test control actions $ u^{\varepsilon}_1, \ldots, u^{\varepsilon}_{n_{\varepsilon}} $ and observes the system's reactions driven by the actual disturbance. In the major initial part of $ [\tau_i, \tau_{i+1}) $ the controlling side implements the useful control action $ u_i $ constructed as the optimal response to the 'surrogate' disturbance for the $ y $-model, whereas the latter is driven by the useful control action $ u_{i-1} $ and 'surrogate' disturbance $ \bar{v}_{i-1} $ formed previously.

The optimal response 
$ u_i $
is found
using Krasovskii's extremal shift principle
(\cite{KraSub88e});
$ u_i $
shifts the
$ y $-model
to a target set at the maximum speed.
The target set
is formed 
in advance and  
comprises 
the histories 
(up to time
$ \tau_i $)
of the uniform limits of the system's motions
corresponding to 
'approximately optimal' 
control quasi-strategies.
The above control process ensures that
the current histories of
both the system's and
$ y $-model's
motions
never abandon small neighborhoods
of the current target sets,
implying that at the final time,
$ \vartheta $,
the value of the cost functional
does not exceed 
$ \Gq(z_0) + \varphi(\varepsilon) $ for some $\varphi(\cdot)$ satisfying $\varphi(\epar)\mathop{\to}\limits_{\epar\to0}0$.
 
Now we turn to formal definitions.
In the construction of the target sets we use 
the system's motions corresponding to 'approximately optimal' control quasi-strategies.
We set
$$
\Uqs(z, \delta) =
\{ \alpha(\cdot) \in \Uqs : \Gq(z, \alpha(\cdot)) \leq \Gq(z) + \delta \}
\quad
(\delta > 0,\ z\in G_0),
$$
$$%\beqnt%\label{z_pr_qs}
\Wq(z)\mydef\bigcap_{\delta>0}\close
\bigcup_{\scriptsize \alpha(\cdot) \in \Uqs(z, \delta)}\Xs(z,\alpha(\cdot));
$$%\eeq
here $\close X$ denotes the closure of a $X\subset C(T; \RA^n)$ in $C(T; \RA^n)$.
For every $ \tau \in T $ the set of the restrictions of all the elements of $ \Wq(z) $ to $ [t_0, \tau] $, denoted by
$ \Wq(z)|_{[t_0,\tau]} $,
will be regarded as the
{\em target set} at time
$ \tau $.
For every
$ \tau \in T $
and every
$y(\cdot)\in C([t_0,\tau],\RA^n)$ we fix a projection  $w(\cdot | \tau, y(\cdot))$
of
$ y(\cdot) $
onto the target set
$ \Wq(z)|_{[t_0,\tau]} $;
thus,
\beq\label{pro_on_Wq}
w(\cdot |\tau, y(\cdot))\in\argmin_{w(\cdot)\in\Wq(y(t_0))|_{[t_0,\tau]}}\|w(\cdot)-y(\cdot)\|_{C}
\eeq
where
$ \| \cdot \|_C $
stands for the
norm in
$ C([t_0,\tau],\RA^n) $.

Fix an  $\epar\in(0,1)$. Fix an $\epar$-net $(u^\epar_j)_{j \in\nint{1}{n_\epar}}$ in $\PC$; thus, $\sup_{u\in\PC}\mymin_{j\in\nint{1}{n_\epar}}\|u-u^\epar_j\|\le\epar$. In the subsequent constructions the elements of $(u^\epar_j)_{j \in\nint{1}{n_\epar}}$ play the role of test control actions mentioned above.

Every $\Part\in\Pparset$ can be "thin out", by withdrawing some elements, to $\Part'\in\Pparset$,  so that the latter satisfy conditions
$\Part'\subseteq\Part$, $\diam{\Part^\prime}/\diamin{\Part^\prime}\le3$ and $\diam{\Part^\prime}\le3\diam{\Part}$.

Let $\Part\mydef\cPart$ be a partition of $T$. Without loss of generality we suppose, that $\diam{\Part}/\diamin{\Part}\le3$.
Denote 
\begin{gather}
\tau_i'\mydef\tau_i-\epar \diamin{\Part}, \quad i\in\nint{1}{(\nPart-1)},\label{diamin}\\
\tau'_{ij}\mydef\tau_{i}'+\frac{j(\tau_{i}-\tau_{i}')}{n_{\epar}},\quad j\in\nint{0}{n_{\epar}}, \quad i\in\nint{1}{(\nPart-1)}\label{dop_inst}
\end{gather}
(thanks to \fref{diamin} $\tau'_{ij}\in(\tau_{i-1},\tau_{i}]$). For every $x(\cdot)\in C(T;\RA^n)$ let
\beq\label{ddd}
d_{ij} (x(\cdot))\mydef\frac{x(\tau'_{ij})-x(\tau'_{i(j-1)})}{\tau'_{ij }-\tau'_{i(j-1)}},\quad j\in\nint{1}{n_{\epar}}, \quad i\in\nint{1}{(\nPart-1)}.
\eeq

Define a full-memory feedback $\Uepar= (\Uep{i}(\cdot))_{i\in\nint{0}{(\nPart-1)}} $ for $\Part$ inductively.
Fix some $u_*\in\PC$, $v_* \in \QC$. 
For every $ x_0 (\cdot) \in C([t_0,\tau_0],\RA^n)$
(recall that
$ \tau_0 = t_0 $) we set
\begin{gather}
\bar{v}_0\mydef v_*,\quad u_0\mydef u_*,\quad  y_0(\tau_0)=z_0,\label{u_0_bar_v_0}\\
\Uep{0}(x_0(\cdot))(t)\mydef
\begin{cases}
u_0,&t\in[\tau_0,\tau_1'),\\
u^\epar_j,&t\in[\tau'_{1(j-1)},\tau'_{1j}), \ j\in\nint{1}{n_\epar}.
\end{cases}\label{Ulp0}
\end{gather}
If for some $i\in\nint{1}{(\nPart-1)}$ elements 
$\bar v_{i-1}=\bar v_{i-1}(x_{i-1}(\cdot))\in\QC$,
$\Uep{(i-1)}(x_{i-1}(\cdot)) \in {\cal U}|_{[\tau_{i-1},\tau_{i}]} $
and
$y_{i-1}(\cdot)=y_{i-1}(\cdot,x_{i-1}(\cdot))\in C([t_0,\tau_{i-1}], \RA^n)$  
(a motion of the
$ y $-model on
$ [t_0, \tau_{i-1}] $)
are defined for all $x_{i-1}(\cdot)\in C([t_0, \tau_{i-1 }], \RA^n)$, then for every $x_{i}(\cdot)\in C([t_0,\tau_i], \RA^n)$ we define
$ y_i(\cdot) =y_{i}(\cdot,x_{i}(\cdot))\in C([t_0,\tau_{i}], \RA^n)$
as the extension of
$ y_{i-1}(\cdot) $
to
$ [t_0, \tau_{i}] $
such that
\begin{multline}
y_{i}(\tau)=y_{i-1}(\tau_{i-1},x_i(\cdot)|_{[t_0,\tau_{i-1}]})\\
+\int_{\tau_{i-1}}^{\tau} f(t,y_{i}(t),\Uep{i-1}(x_i(\cdot)|_{[t_0,\tau_{i-1}]})(\tau_{i-1}),\bar{v}_{i-1}(x_i(\cdot)|_{[t_0,\tau_{i-1}]}))dt,\\
\tau\in[\tau_{i-1},\tau_i],\label{mov_y_def}
\end{multline}
and set
\begin{gather}
\bar{v}_{i}\in\argmin_{v\in\QC}\mymax_{j\in\nint{1}{n_\epar}}\|d_{ij}(x_i(\cdot))-f(\tau_i,x_{i}(\tau_i),u_j^\epar,v)\|,
\label{bar_v_def}\\
u_{i}\in\argmin_{u \in\PC}\SP{y_{i}(\tau_i)-w(\tau_i\mid\tau_i, y_i(\cdot))}{f(\tau_i,y_{i}(\tau_i),u,\bar{v}_{i})},\label{U_bPart_def}\\
\Uep{i}(x_i(\cdot))(t)\mydef
\begin{cases}
u_i,&t\in[\tau_i,\tau_{i+1}'),\\
u^\epar_j,&t\in[\tau'_{(i+1)(j-1)},\tau'_{(i+1)j}), j\in\nint{1}{n_\epar}.
\end{cases}
\label{Ulpi}
\end{gather}
The full-memory feedback $\Uepar$ is defined for an arbitrary partition $\Part\in\Pparset$. 
Thus, the full-memory strategy $\Ue\mydef (\Uepar)_{\Part\in\Pparset}$ is defined.

%%%%%%%%%%%%%%%%%%%%%
\begin{theorem}\label{teo_Q_EQ_C}
For all $z_0\in G_0$ the following relations hold true:
\beq\label{Ue_approx_gamma_ps}
\limsup_{\epar\to0}\Gc(z_0,\Ue)\le\Gq(z_0),
\eeq
%%%
\beq\label{gamma_pr_eq_low_game_value}
\Gp(z_0) = \Gc(z_0) = \Gq(z_0).
\eeq
\end{theorem}
%%%%%%%%%%%%%%%%%%%%%%

A proof is given in the next section.

\begin{example}
Let system
(\ref{sys})
has the form
%%%%%%%
\beq
\label{sys_examp_2x2}
\begin{cases}
\dot x_1(\tau)=u_1(\tau)\cdot v_1(\tau),\\
\dot x_2(\tau)=\max\{0, x_1(\tau))\}\cdot u_2(\tau)\cdot  v_2(\tau),\\
(x_1(0),x_2(0))=(0,0),
\end{cases}
\quad
\tau\in T\mydef[0,1],
\quad
G_0=\{(0,0)\},
\eeq
%%%%%%%
$$
u_1(\tau), u_2(\tau)\in  [-1,1],\quad v_1(\tau),v_2(\tau)\in \{-1,1\},
$$
and the cost functional be given by $\gamma(x(\cdot))\mydef x_2(1)$ 
$ (x(\cdot) = (x_1(\cdot), x_2(\cdot)) \in C(T,\RA^2)) $.
With an appropriate choice of
$ G $,
system \fref{sys_examp_2x2} 
satisfies all the assumptions imposed
earlier on system
(\ref{sys});
therefore,
Theorem
\ref{teo_Q_EQ_C}
holds,
implying the
full-memory control strategies
are non-improvable
against both the
$ L_2 $-compactly constrained and open-loop disturbances.
On the other hand, system
(\ref{sys})
does not satisfy the 
conditions sufficient for the non-improvability
of the full-memory control strategies
against the
$ L_2 $-compactly constrained disturbances,
which are given in
\cite{Kry91} 
(Theorem 9.1)
and in
\cite{Ser_DAN2013et}
(Theorem 2).
One can easily find that
$
\Gc((0,0))=\Gp((0,0))=\Gq((0,0))=-0.5 
$
and
$\Gs((0,0))=0$. 
\end{example}

%%%%%%%%%%%%%%%%%%%
\section {Proof of Theorem \ref{teo_Q_EQ_C}}
%%%%%%%%%%%%%%%%%%%
Since
(\ref{Ue_approx_gamma_ps})
and
(\ref{neq_gam})
imply
(\ref{gamma_pr_eq_low_game_value}),
it is sufficient to prove
(\ref{Ue_approx_gamma_ps}).

Introduce some definitions and notations.

For every $p,q:T\mapsto S$ where 
$ S $ is a nonempty set
and every
$t'\in[t,\vartheta] $ we denote
$$%\beqnt
(p,q)_{t'}(\tau)\mydef
\begin{cases}
p(\tau),&\tau\in[t_0,t'),\\
q(\tau),&\tau\in[t',\vartheta].
\end{cases}
$$%\eeq

After \cite{KraSub88e}, we call a set $W\subseteq C({T};\RA^n)$  $u$-{\em stable}, if for any $ [\tau_*,\tau^*]\subseteq{T}$, $v_*\in\QC$, $x_*(\cdot)\in W|_{[t_0,\tau_*]}$ there is a solution $x(\cdot)$ of the differential inclusion
$$%\beqnt%\label{incl_u_stab}
\begin{cases}
\dot x(\tau)\in\FC_u(\tau,x(\tau),v_*) &\text{for \aee\ }\tau\in[\tau_*,\tau^*],\\
x(\tau_*)=x_*(\tau_*)
\end{cases}
$$%\eeq
such that  $ (x_*,x)_{\tau_*}(\cdot)\in W|_{[t_0,\tau^*]}$;
here
$
\FC_u(\tau,x,v)
$
is the closed convex hull 
of
the set
$ \{ f(\tau,x,u,v) : u\in\PC\} $
in
$ \RA^n $.

We set
$$
I_c(\tau)\mydef\left\{(a,b)\in\RA^2 :  b>a, \frac{\max\{|b-\tau|,|a-\tau|\}}{b-a}\le c\right\}\quad 
(c \geq 1/2,\quad \tau\in T)
$$
and
for an arbitrary measurable set $A\subseteq\RA $ denote
%%%%%
\beq\label{def_Aprim}
A'_c\mydef\left\{\tau\in\RA \ : \lim_{a,b\to\tau\atop (a,b)\in I_c(\tau)}\frac{\mes(A\cap[a,b])}{b-a}=1\right\};
\eeq
%%%%%
here and below
$ \lambda $
is the Lebesgue measure.
\begin{lemma}\label{teo_lebeg_app}
For any measurable set $A\subseteq\RA$ and any $c\geq 1/2 $
it holds that 
\beq\label{eq_lebeg_app}
\mes(A\triangle A'_c)\mydef\mes((A\backslash A'_c)\cup(A'_c\backslash A))=0.
\eeq
\end{lemma}
%%%%%%
\begin{remark}
For $c=1/2$ the 
second equality in
(\ref{eq_lebeg_app})
was proved in \cite{Oxtobye} (Theorem 3.20) and in \cite{Natansonet} (ch.IX, \S6). 
The lemma follows from \cite[Theorem 7.10]{Rudin1986}. We give simplified proof.
\end{remark}

\begin{proof} 
Let $t_0\in\RA$ and absolutely continuous function $\Lambda:\RA\mapsto\RA$ has the form
$$
\Lambda(\tau)\mydef\int\limits_{t_0}^\tau\chi_A(s)\,ds,\quad
\chi_A(s)\mydef
\begin{cases}
1,&s\in A,\\
0,&s\not\in A
\end{cases}
$$
where $ \chi_A(\cdot)$ denotes the characteristic function of the set $A$. The Lebesgue differentiation theorem implies that for \aee $\tau\in\RA$ the function $\Lambda(\cdot)$ has the derivative at the point $\tau$, that equals to the value $\chi_A(\tau)$. Let ${\cal D}_A\subset\RA$ denotes the set of all Lebesgue points of $\Lambda(\cdot)$ points of differentiability of the function $\Lambda(\cdot)$. Thus, for every $\tau\in{\cal D}_A$ there exists a function $O_\tau(\cdot):\RA\mapsto\RA$, such that $\lim_{\delta\to0}O_\tau(\delta)=0$ and for any $a,b\in\RA$ the equalities hold
\begin{gather*}
\Lambda(a)=\Lambda(\tau)+(a-\tau)\chi_A(\tau)+(a-\tau)O_\tau(a-\tau),\\
\Lambda(b)=\Lambda(\tau)+(b-\tau)\chi_A(\tau)+(b-\tau)O_\tau(b-\tau).
\end{gather*}
Subtracting the second equation from the first one and dividing by $b-a$ we obtain the relations
\begin{multline*}
\left|\frac{\Lambda(b)-\Lambda(a)}{b-a}-\chi_A(\tau)\right|=\left|\frac{(b-\tau)O_\tau(b-\tau)-(a-\tau)O_\tau(a-\tau)}{b-a}\right|\\
\le\frac{\max\{|b-\tau|,|a-\tau|\}}{b-a}(|O_\tau(a-\tau)|+|O_\tau(b-\tau)|)\\
\le c(|O_\tau(a-\tau)|+|O_\tau(b-\tau)|)
\end{multline*}
 for any $a,b\in I_c(\tau)$. Hence, the equality holds
$$
\lim_{a,b\to\tau\atop (a,b)\in I_c(\tau)}\frac{\Lambda(b)-\Lambda(a)}{b-a}=\chi_A(\tau),\qquad  \tau\in{\cal D}_A.
$$
Taking into account the definition of the function $\Lambda(\cdot)$, the last statement can be rewritten as
$$
\lim_{a,b\to\tau\atop (a,b)\in I_c(\tau)}\frac{\mes(A\cap[a,b])}{b-a}=\chi_A(\tau),\qquad\tau\in{\cal D}_A.
$$
This relation shows, that for \aee $\tau\in A$  the inclusion $\tau\in A'_c$ holds and, vice versa, for \aee $ \tau\in A'_c$ the relation $\tau\in A$ is fulfilled.
\end{proof}

We set $X(G_0) = \{ x(\cdot, t_0, z_0, u(\cdot), v(\cdot)) : z_0 \in G_0, \ u(\cdot) \in {\cal U}, \ v(\cdot) \in {\cal V} \}$.

%%%%%%%%%%%%%
\begin{lemma}\label{lem_local_disturb_approx_1}
Let $c\geq 1/2 $ and  $v(\cdot)\in\Vpr$. Then 
\beq\label{bar_v_est_1}
\lim_{a,b\to\tau\atop{(a,b)\in I_c(\tau)\atop a,b\in T}}\sup_{\scriptsize u\in\PC\atop x(\cdot)\in X(G_0)}
\Bigl\|(b-a)^{-1}\int\limits_{[a,b]}f(s,x(s),u,v(s))\,ds-f(\tau,x(\tau),u,v(\tau))\Bigr\|=0
\eeq
for
\aee $ \tau\in{T} $.
\end{lemma}
%%%%%%%%%%%%%
\begin{proof}
Fix a $c\geq 1/2 $, $v(\cdot)\in\Vpr$ and note that by \fref{kappa_G} for all $x(\cdot)\in X(G_0)$ we have
\beq\label{bar_v_est_3}
\sup_{s,\tau\in T}\|x(\tau)-x(s)\|\le\varkappa(G)|\tau-s|.
\eeq
Choose any $\varepsilon>0$. By Luzin's theorem (see \cite[Ch.4]{Natansonet}) 
there is a closed measurable set $E_\varepsilon\subseteq{T}$ such that
\beq\label{bar_v_est_4}
\mes({T}\backslash E_\varepsilon)\le\varepsilon,\quad v(\cdot)\in C(E_\varepsilon,\RA^q).
\eeq
%%%
Let $E_\varepsilon'$ be the set of all density points of $E_\varepsilon$ (see \fref {def_Aprim}). From the closedness of $E_\varepsilon$ 
it follows that $E_\varepsilon'\subseteq E_\varepsilon$. 
By the Lemma \ref{teo_lebeg_app}
applied to
$E_\varepsilon$ we have 
\beq
\mes({T}\backslash E_\varepsilon')\le\varepsilon. 
\label{mmm}
\eeq
The continuity of the right-hand side of equation \fref{sys} in $G\times\PC\times\QC $, the compactness of the latter set and relations \fref{bar_v_est_3}, \fref{bar_v_est_4} imply 
that the functions $ s\mapsto f(s,x(s),u,v(s)): E_\varepsilon \mapsto \RA^n$ are equicontinuous 
with respect to $u\in\PC$, $x(\cdot)\in X(G_0)$, that is, there exists a $\varphi_\varepsilon(\cdot):(0,+\infty)\mapsto(0,+\infty)$ (depending on $E_\varepsilon$) such that $\mylim_{\delta\rightarrow 0}\varphi_\varepsilon(\delta)=0$ and for all $s,\tau\in E_\varepsilon$ it holds that
\begin{equation}
\sup_{\scriptsize u\in\PC\atop x(\cdot)\in X(G_0)}\|f(s,x(s),u,v(s))-f(\tau,x(\tau),u,v(\tau))\|\le\varphi_\varepsilon(|s-\tau|).
\label{bar_v_est_5}
\end{equation}

For all $a,b,\tau\in T$, $a<b$, we have 
\begin{multline*}%\label{bar_v_est_13}
\sup_{\scriptsize u\in\PC\atop x(\cdot)\in X(G_0)}\Bigl\|(b-a)^{-1}\int\limits_{[a,b]}f(s,x(s),u,v(s))\,ds-f(\tau,x(\tau),u,v(\tau))\Bigr\|\\
=\sup_{\scriptsize u\in\PC\atop x(\cdot)\in X(G_0)}(b-a)^{-1}\Bigl\|\int\limits_{[a,b]}f(s,x(s),u,v(s))\,ds-\int\limits_{[a,b]}f(\tau,x(\tau),u,v(\tau))ds\Bigr\|\\
\le\sup_{\scriptsize u\in\PC\atop x(\cdot)\in X(G_0)}(b-a)^{-1}\int\limits_{[a,b]}\bigl\|f(s,x(s),u,v(s))-f(\tau,x(\tau),u,v(\tau))\bigr\|\,ds.
\end{multline*}
%%%
Let $\tau\in E'_\varepsilon$. We decompose the last integral into the sum of two ones, using the set $ E_\varepsilon$, 
and apply to the first term the  estimate \fref{bar_v_est_5}. 
Thus, we continue the calculations as follows:
\begin{multline*}%\label{bar_v_est_13}
=\sup_{\scriptsize u\in\PC\atop x(\cdot)\in X(G_0)}(b-a)^{-1}\int\limits_{[a,b]\cap E_\varepsilon}\left\|f(s,x(s),u,v(s))-f(\tau,x(\tau),u,v(\tau))\right\|\,ds\\
+\sup_{\scriptsize u\in\PC\atop x(\cdot)\in X(G_0)}(b-a)^{-1}\int\limits_{[a,b]\setminus  E_\varepsilon}\left\|f(s,x(s),u,v(s))-f(\tau,x(\tau),u,v(\tau))\right\|\,ds\\
\le\varphi_\varepsilon(\max\{|a-\tau|,|b-\tau|\})+2\varkappa(G)\frac{\mes([a,b]\setminus E_\varepsilon)}{b-a}.
\end{multline*}
%%%
Therefore, we get
\begin{multline*}%\label{bar_v_est_13}
\sup_{\scriptsize u\in\PC\atop x(\cdot)\in X(G_0)}\Bigl\|(b-a)^{-1}\int\limits_{[a,b]}f(s,x(s),u,v(s))\,ds-f(\tau,x(\tau),u,v(\tau))\Bigr\|\\
\le\varphi_\varepsilon(\max\{|a-\tau|,|b-\tau|\})+2\varkappa(G)\frac{\mes([a,b]\setminus E_\varepsilon)}{b-a}
\end{multline*}
where the right-hand side (since $\tau \in E'_{\varepsilon}$) tends to zero as $(a,b)\in I_c(\tau)$ and $a,b\to\tau$.
Thus, for all
$\tau \in E'_{\varepsilon} $
we have
(\ref{bar_v_est_1}).
In view of
(\ref{mmm})
and the arbitrary choice of
$ \varepsilon > 0 $
we get that
(\ref{bar_v_est_1})
holds for
\aee
$ \tau \in T $.
\end{proof}

\begin{lemma}\label{lem_z_pr_na_u_stab}
For any position $z\in G_0$ the set $\Wq(z)$ is compact in $ C (T, \RA ^ n) $, $u$--stable, change upper semicontinuously by inclusion with respect to parameter $z\in G_0$ and satisfies the relations
\beq\label {z_pr_na_gives_opt_guarantee}
\max_{w(\cdot)\in\Wq(z)}\qndx(w(\cdot)) = \Gq (z), \quad z\in\Wq(z)|_{t_0}, \ z\in G_0.
\eeq
\end {lemma}
The proof of Lemma \ref{lem_z_pr_na_u_stab} is standard for  theory of closed-loop differential games and follows \cite[Lemma 5.1]{Ser_MGTA2012e}. 

The inclusion $z\in\Wq(z)|_{t_0}$ and $u$--stability of the set $ \Wq(z)$ imply the inequality $\Wq(z)|_{[t_0,\tau]}\neq\varnothing$, $\tau\in{T}$. And in view of the compactness in $C(T;\RA^n)$ of the set $\Wq(z)$, projections \fref{pro_on_Wq} are defined correctly.

Now we fix an $\epar>0$, a $z_0\in G_0$ and a motion ${x}_0(\cdot)\in\Xc(z_0,\Ue)$. 
By the definition of $ \Xc(z_0,\Ue)$ there exist a $\VB\in\cVpr$
and a sequence
$(z_{0k},v_k(\cdot),\Partk,\Uepark)_{k=1}^{\infty} $ 
in $G_0\times \VB \times\Pparset\times\Ue$ 
such that
$\lim_{k\rightarrow\infty}z_{0k}=z_0$, $\lim_{k\rightarrow\infty}\diam{\Partk}=0$
and
%%%%%%
\begin{equation}
x_k(\cdot) \rightarrow x_0(\cdot) \quad \mbox{in} \quad {C(T;\RA^n)},
\label{xk_to_x0}
\end{equation}
where $x_k(\cdot)\mydef x(\cdot,z_{0k},\Uepark,v_k(\cdot))$.
%%%%%%%
Since
sequence
$(v_k(\cdot))_{k=1}^{\infty} $ 
lies in 
$ \VB\in \mbox{comp}_{L_2}({\cal V}) $,
it has a subsequence
convergent in
$ L_2(T;\RA^q) $.
With no loss of generality
(selecting, if necessary, a subsequence),
we get
%%%%%%%
\begin{equation}
v_k(\cdot) \rightarrow v_0(\cdot) \quad \mbox{in} \quad {L_2(T;\RA^n)}.
\label{vk_to_v0_L2}
\end{equation}
%%%%%%%
Obviously, $ v_0(\cdot) \in {\cal V}. $
%%%%%%%
Using \fref{vk_to_v0_L2} and a convergence property of measurable functions (see \cite[Theorem I.4.8]{Wargae}),
we find that,
with no loss of generality
(if necessary, we select a subsequence),
it holds that
%%%
\beq\label{vk_to_v0}
\lim_{k\rightarrow\infty}v_k(\tau)=v_0(\tau),\quad \text{for \aee} \quad \tau\in{T}.
\eeq
%%%
For every $k\in \{1,2,\ldots \} $ denote
$$%\beqnt%\label{sbs_x_u}
u_k(\cdot)\mydef u(\cdot,z_{0k},\Uepark,v_k(\cdot)),
$$%\eeq
let $\Partk\mydef\cPartk$, 
\beq\label{thin-out-partk}
\diam{\Partk}/\diamin{\Partk}\le3
\eeq 
and, in accordance with \fref{dop_inst},  \fref{ddd}, set
\beq\label{dop_inst_k}
\tau_{ki}' = \tau_{ki}-\epar\diamin{\Partk},
\quad
\tau'_{kij} =  \tau_{ki}'+\frac{j(\tau_{ki}-\tau_{ki}')}{n_{\epar}},  \quad j\in\nint{0}{n_{\epar}},\quad i\in\nint{1}{(\nPartk-1)},
\eeq
$$
d_{kij} (x(\cdot))\mydef\frac{x(\tau'_{kij})-x(\tau'_{ki(j-1)})}{\tau'_{kij }-\tau'_{ki(j-1)}},\quad j\in\nint{1}{n_{\epar}}, \quad i\in\nint{1}{(\nPartk-1)}.
$$
Take a $k\in \{1,2,\ldots \} $. By the definition of the full-memory feedback $\Uepark= (\Uepk{i})_{i\in\nint{0}{(\nPart-1)}}$ for $\Partk$ (see \fref{u_0_bar_v_0}, \fref{Ulp0}, \fref{mov_y_def}, \fref{bar_v_def}, \fref{U_bPart_def}, \fref{Ulpi}) we also determine the motion $y_k(\cdot)$ of $y$--model, associated with $x_k(\cdot)$, control $\bar u_k(\cdot)$ and disturbance $\bar v_k(\cdot)$, that operate in the motion  $y_k(\cdot)$
\begin{gather}
y_k(\tau)=z_{0k}+\int\limits_{t_0}^\tau f(s,y_k(s),\bar u_k(s),\bar v_k(s))\,ds,\nonumber\\
\bar v_k(t)\mydef\bar v_{ki_t},\quad \bar u_k(t)\mydef u_{ki_t},\qquad \tau,s\in T,\label{sbs_y_bu_bv_inst}\\
\bar{v}_{k0}\mydef v_*,\quad u_{k0} \mydef u_{*},\quad  y_{k}(\tau_0)=z_{k0},\nonumber\\%\label{u_0_bar_v_0-1}
\Uepk{i}(x_{k}(\cdot)|_{[t_0,\tau_{k0}]})(t) =
\begin{cases}
u_{k0} ,&t\in[\tau_{k0},\tau_{k1}'),\\
u^\epar_j,&t\in[\tau'_{k1(j-1)},\tau'_{k1j}), \ j\in\nint{1}{n_\epar},
\end{cases}\nonumber\\%\label{Ulp0-1}
\bar{v}_{ki}\in\argmin_{v\in\QC}\mymax_{j\in\nint{1}{n_\epar}}\|d_{kij}(x_{k}(\cdot))-f(\tau_{ki},x_{k}(\tau_{ki}),u_j^\epar,v)\|,\nonumber\\%\label{bar_v_def}
u_{ki}\in\argmin_{u \in\PC}\SP{y_{k}(\tau_{ki})-w(\tau_{ki}\mid\tau_{ki}, y_{k}(\cdot))}{f(\tau_{ki},y_{k}(\tau_{ki}),u,\bar{v}_{ki})},\nonumber\\%\label{U_bPart_def}
\Uepk{i}(x_k(\cdot)|_{[t_0,\tau_{ki}]})(t)\mydef
\begin{cases}
u_{ki},&t\in[\tau_{ki},\tau_{k(i+1)}'),\\
u^\epar_j,&t\in[\tau'_{(ki+1)(j-1)},\tau'_{(ki+1)j}), j\in\nint{1}{n_\epar},
\end{cases}\nonumber%\label{Ulpi}
\end{gather}
where $u_{*}\in\PC$, $v_{*} \in \QC$,  $x_{k}(\cdot)|_{[t_0,\tau_{ki}]}$ is the restriction of $x_k(\cdot)$ to $[t_0, \tau_{ki}]$ ($[t_0, \tau_{k0}] = \{t_0\}$).

For all $(t, x)\in G$ and the $\epar$--net $(u_j^\epar)_{j\in\nint{1}{n_\epar}}$, we introduce the quotient set $\QC_{tx\epar}$ of the set $\QC$, generated by the equivalence relation $\mysim{tx\epar}$:
$$%\beqnt%\label{FM_Qe}
(v_1\mysim{tx\epar} v_2)\Leftrightarrow((\forall j\in\nint{1}{n_\epar}) f(t,x,u_j^\epar,v_1)=f(t,x,u_j^\epar,v_2)).
$$%\eeq
By the similar way for all $(t, x,u)\in G\times\PC$ we define the quotient set $\QC_{txu}$ of the set $\QC$, generated by the equivalence relation $\mysim{txu}$:
$$%\beqnt%\label{FM_Qu}
(v_1\mysim{txu} v_2)\Leftrightarrow(f(t,x,u,v_1)=f(t,x,u,v_2)).
$$%\eeq
For all $t\in T$ denote $q^\epar_t$, $q^u_t$ the equivalence classes, that satisfy the conditions 
$$
v_0(t)\in q^\epar_t\in\QC_{t x_0(t)\epar},\quad v_0(t)\in q^u_t\in\QC_{t x_0(t)u}.
$$
Then the equalities hold 
\beq\label{qe_eq_c_quej}
q^\epar_t=\bigcap_{j\in\nint{1}{n_\epar}}q^{u^\epar_j}_t\qquad t\in T.
\eeq

Given a metric space $(X,\rho)$ and non--empty subsets  $A,B\subset X$, denote $\dist{X}(A,B)$  Hausdorff semi-metric from $B$ to $A$:
$$
\dist{X}(A,B)\mydef \sup_{a\in A}\myinf_{b\in B}\rho(a,b).
$$

\begin{lemma}\label{lem_bvk_to_qtaue}
For \aee $\tau\in T$ the following equality holds
\beq\label{bvk_to_qtaue}
\lim_{k\rightarrow\infty}\dist{\RA^q}(\{\bar v_k(\tau)\},q^\epar_\tau) = 0.
\eeq
\end{lemma}

\begin{proof}
1. For all $ k\in\NA$, $j\in\nint{1}{n_{\epar}}$ and $\tau\in T$ denote
\begin{gather*}
x_{k0}(\cdot)\mydef x(\cdot,t_0,z_{0k},u_k(\cdot),v_0(\cdot)),\\
D_{kj}(\tau)\mydef d_{ki_\tau j}(x_k(\cdot))=\int\limits_{\tau'_{ki_\tau(j-1)}}^{\tau'_{ki_\tau j}}\frac{f(s,x_{k}(s),u_k(s),v_k(s))}{\tau'_{ki_\tau j}-\tau'_{ki_\tau(j-1)}}ds.
\end{gather*}
Note, that for all $k\in\NA$, $j\in\nint{1}{n_{\epar}}$, and $\tau\in T$ the following relations hold
\begin{gather*}
\bar{v}_k(\tau)\in\argmin_{v\in\QC}\mymax_{j\in\nint{1}{n_{\epar}}}\|D_{kj}(\tau)-f(\tau_{ki_\tau},x_{k}(\tau_{ki_\tau}),u_j^{\epar},v)\|,\\
u_k(\tau)=u_j^{\epar},\qquad \tau\in[\tau'_{ki_\tau(j-1)},\tau'_{ki_\tau j}).
\end{gather*}
In addition, due to convergence \fref{vk_to_v0}, the equality holds
$$%\beqnt%\label{xk0_to_xkk}
\lim_{k\rightarrow\infty}\|x_k(\cdot)-x_{k0}(\cdot)\|_{C({T};\RA^n)}=0.
$$%\eeq

2. We estimate the value
\begin{multline*}
\Bigl\|D_{kj}(\tau)-f(\tau,x_k(\tau),u_j^{\epar},v_k(\tau))\Bigr\|\\
=\Bigl\|\int\limits_{\tau'_{ki_\tau(j-1)}}^{\tau'_{ki_\tau j}}\frac{f(s,x_{k}(s),u_j^{\epar},v_k(s))}{\tau'_{ki_\tau j}-\tau'_{ki_\tau(j-1)}}ds-f(\tau,x_k(\tau),u_j^{\epar},v_k(\tau))\Bigr\|\\
\le\int\limits_{\tau'_{ki_\tau(j-1)}}^{\tau'_{ki_\tau j}}\Bigl\|\frac{f(s,x_k(s),u_j^{\epar},v_k(s))-f(s,x_k(s),u_j^{\epar},v_0(s))}{\tau'_{ki_\tau j}-\tau'_{ki_\tau(j-1)}}\Bigr\|ds\\
+\int\limits_{\tau'_{ki_\tau(j-1)}}^{\tau'_{ki_\tau j}}\Bigl\|\frac{f(s,x_k(s),u_j^{\epar},v_0(s))-f(s,x_{k0}(s),u_j^{\epar},v_0(s))}{\tau'_{ki_\tau j}-\tau'_{ki_\tau(j-1)}}\Bigr\|ds\\
+\Bigl\|\int\limits_{\tau'_{ki_\tau(j-1)}}^{\tau'_{ki_\tau j}}\frac{f(s,x_{k0}(s),u_j^{\epar},v_0(s))}{\tau'_{ki_\tau j}-\tau'_{ki_\tau(j-1)}}ds-f(\tau,x_{k0}(\tau),u_j^{\epar},v_0(\tau))\Bigr\|\\
+\|f(\tau,x_{k0}(\tau),u_j^{\epar},v_0(\tau))-f(\tau,x_k(\tau),u_j^{\epar},v_k(\tau))\|.
\end{multline*}
We use the properties of the uniform continuity and Lipschitz property of $f(\cdot)$ in area $G\times\PC\times\QC$ (continuing evaluation):
\begin{multline*}
\le\int\limits_{\tau'_{ki_\tau(j-1)}}^{\tau'_{ki_\tau j}}\frac{\mfv(\|v_k(s)-v_0(s)\|)+L_f(G)\|x_k(s)-x_{k0}(s)\|}{\tau'_{ki_\tau j}-\tau'_{ki_\tau(j-1)}}ds\\
+\Bigl\|\int\limits_{\tau'_{ki_\tau(j-1)}}^{\tau'_{ki_\tau j}}\frac{f(s,x_{k0}(s),u_j^{\epar},v_0(s))}{\tau'_{ki_\tau j}-\tau'_{ki_\tau(j-1)}}ds-f(\tau,x_{k0}(\tau),u_j^{\epar},v_0(\tau))\Bigr\|\\
+L_f(G)\|x_k(\tau)-x_{k0}(\tau)\|+\mfv(\|v_k(\tau)-v_0(\tau)\|)\\
\le\Bigl\|\int\limits_{\tau'_{ki_\tau(j-1)}}^{\tau'_{ki_\tau j}}\frac{f(s,x_{k0}(s),u_j^{\epar},v_0(s))}{\tau'_{ki_\tau j}-\tau'_{ki_\tau(j-1)}}ds-f(\tau,x_{k0}(\tau),u_j^{\epar},v_0(\tau))\Bigr\|\\
+\int\limits_{\tau'_{ki_\tau(j-1)}}^{\tau'_{ki_\tau j}}\frac{\mfv(\|v_k(s)-v_0(s)\|)}{\tau'_{ki_\tau j}-\tau'_{ki_\tau(j-1)}}ds+2L_f(G)\|x_k(\cdot)-x_{k0}(\cdot)\|_{C(T;\RA^n)}+\mfv(\|v_k(\tau)-v_0(\tau)\|);
\end{multline*}
here $\mfv(\cdot)$ denotes the modulus of continuity of $f(\cdot)$  in the forth argument:
$$%\beqnt%\label{mfv}
\mfv(\delta)\mydef\max_{|v-v'|\le\delta\atop{(\tau,x)\in G\atop u\in\PC, v,v'\in\QC}}\|f(\tau,x,u,v)-f(\tau,x,u,v')\|,\quad\lim_{\delta\rightarrow+0}\mfv(\delta)=0.
$$%\eeq
Thus,
\begin{multline}\label{Dkj_too_fvk}
\Bigl\|D_{kj}(\tau)-f(\tau,x_k(\tau),u_j^{\epar},v_k(\tau))\Bigr\|\\
\le\Bigl\|\int\limits_{\tau'_{ki_\tau(j-1)}}^{\tau'_{ki_\tau j}}\frac{f(s,x_{k0}(s),u_j^{\epar},v_0(s))}{\tau'_{ki_\tau j}-\tau'_{ki_\tau(j-1)}}ds-f(\tau,x_{k0}(\tau),u_j^{\epar},v_0(\tau))\Bigr\|+\Psi_{2kj}(\tau),
\end{multline}
where
\begin{multline*}
\Psi_{2kj}(\tau)\mydef\int\limits_{\tau'_{ki_\tau(j-1)}}^{\tau'_{ki_\tau j}}\frac{\mfv(\|v_k(s)-v_0(s)\|)}{\tau'_{ki_\tau j}-\tau'_{ki_\tau(j-1)}}ds\\
+2L_f(G)\|x_k(\cdot)-x_{k0}(\cdot)\|_{C(T;\RA^n)}+\mfv(\|v_k(\tau)-v_0(\tau)\|).
\end{multline*}

By definition \fref{dop_inst_k} of additional points $\tau'_{ki_\tau(j-1)}, \tau'_{ki_\tau j}$ we have 
\begin{gather*}
\max\{|\tau'_{ki_\tau(j-1)}-\tau|,|\tau'_{ki_\tau j}-\tau|\}=\tau-\tau'_{ki_\tau(j-1)}\le\diam{\Partk}+\epar\diamin{\Partk},\\
\tau'_{ki_\tau j}-\tau'_{ki_\tau(j-1)}=\frac{\epar\diamin{\Partk}}{n_\epar},\quad \tau\in{T},\ k\in\NA,\ j\in\nint{1}{n_{\epar}}.
\end{gather*}
Then, the condition \fref{thin-out-partk} implies the inequalities
$$
\frac{\max\{|\tau'_{ki_\tau(j-1)}-\tau|,|\tau'_{ki_\tau j}-\tau|\}}{\tau'_{ki_\tau j}-\tau'_{ki_\tau(j-1)}}\le\frac{n_\epar(\diam{\Partk}+\epar\diamin{\Partk})}{\epar\diamin{\Partk}}\le \frac{n_\epar(3+\epar)}{\epar} \mydef c_\epar\ge1/2,
$$
for all $\tau\in{T}$, $k\in\NA$, $j\in\nint{1}{n_{\epar}}$. So, the inclusions take place
$$
\tau'_{ki_\tau(j-1)},\tau'_{ki_\tau j}\in I_{c_\epar}(\tau),\quad \tau\in{T},\ k\in\NA,\ j\in\nint{1}{n_{\epar}}.
$$
Hence, by Lemma \ref{lem_local_disturb_approx_1}, we get the convergence
\beq\label{eq_lim_1}
\lim_{k\rightarrow\infty}\max_{j\in\nint{1}{n_{\epar}}}\bigg\|\int\limits_{\tau'_{ki_\tau(j-1)}}^{\tau'_{ki_\tau j}}\frac{f(s,x_{k0}(s),u_j^{\epar},v_0(s))}{\tau'_{ki_\tau j}-\tau'_{ki_\tau(j-1)}}ds-f(\tau,x_{k0}(\tau),u_j^{\epar},v_0(\tau))\bigg\|=0
\eeq
 for \aee $\tau\in{T}$. 

Applying Lemma \ref{teo_lebeg_app} to the integral item in $\Psi_{2kj}$ and taking into account the convergences \fref{xk_to_x0} and \fref{vk_to_v0}, one can verify the equalities
\beq\label{Psi2_to_0}
\lim_{k\to\infty}\Psi_{2kj}(\tau)=0,\quad\text{for \aee $\tau\in{T}$}.
\eeq

Relations \fref {Dkj_too_fvk}, \fref{eq_lim_1}, \fref{Psi2_to_0} lead to equalities
\beq\label{Dkj_to_fvk}
\lim_{k\to\infty}\max_{j\in\nint{1}{n_{\epar}}}\Bigl\|D_{kj}(\tau)-f(\tau,x_k(\tau),u_j^{\epar},v_k(\tau))\Bigr\|=0,\quad\text{for \aee $\tau\in{T}$},
\eeq
that, in its turn, with the convergences \fref {xk_to_x0}, \fref{vk_to_v0} and the uniform continuity of $f(\cdot)$ in $ G\times\PC\times\QC$, give the equalities
\beq\label{Dkj_to_fv0}
\lim_{k\to\infty}\max_{j\in\nint{1}{n_{\epar}}}\Bigl\|D_{kj}(\tau)-f(\tau,x_0(\tau),u_j^{\epar},v_0(\tau))\Bigr\|=0,\quad\text{for \aee $\tau\in{T}$}.
\eeq

3. The continuity of $f(\cdot)$ in the area $G\times\PC\times\QC$ and equicontinuity of the sequence $(x_k(\cdot))_{k\in\NA}$ implies the existence of a function $\psi(\cdot): (0,1)\mapsto(0,+\infty)$ of the form
$$%\beqnt
\psi(\delta)\mydef\sup_{u\in\PC,v\in\QC, k\in\NA\atop \tau,\tau'\in{T}, |\tau-\tau'|\le\delta}\|f(\tau',x_k(\tau'),u,v)-f(\tau,x_k(\tau),u,v)\|\le\psi(|\tau'-\tau|).
$$%\eeq
and such that $\lim_{\delta\to+0}\psi(\delta)=0$.
Hence, for any $k\in\NA$ and $\tau\in T$ we obtain the inequalities
\begin{multline}\label{Dkj_to_fbvk}
\max_{j\in\nint{1}{n_{\epar}}}\|D_{kj}(\tau)-f(\tau,x_k(\tau),u_j^{\epar},\bar v_k(\tau))\|\\
\le\max_{j\in\nint{1}{n_{\epar}}}\|D_{kj}(\tau)-f(\tau_{ki_\tau},x_k(\tau_{ki_\tau}),u_j^{\epar},\bar v_k(\tau_{ki_\tau})))\|+\psi(|\tau-\tau_{ki_\tau}|)\\
\le\min_{v\in\QC}\max_{j\in\nint{1}{n_{\epar}}}\|D_{kj}(\tau)-f(\tau_{ki_\tau},x_k(\tau_{ki_\tau}),u_j^{\epar}, v))\|+\psi(|\tau-\tau_{ki_\tau}|)\\
\le\min_{v\in\QC}\max_{j\in\nint{1}{n_{\epar}}}\|D_{kj}(\tau)-f(\tau,x_k(\tau),u_j^{\epar}, v))\|+2\psi(|\tau-\tau_{ki_\tau}|)\\
\le\max_{j\in\nint{1}{n_{\epar}}}\|D_{kj}(\tau)-f(\tau,x_k(\tau),u_j^{\epar},v_k(\tau)))\|+2\psi(|\tau-\tau_{ki_\tau}|).
\end{multline}
The first inequality is also based on the identity $\bar v_k(\tau)=\bar v_k(\tau_{ki_\tau})$, $\tau\in T$, $k\in\NA$ (see \fref{sbs_y_bu_bv_inst}).

From \fref{Dkj_to_fvk}, \fref {Dkj_to_fv0}, \fref{Dkj_to_fbvk} and the convergences \fref{xk_to_x0}, \fref{vk_to_v0} the equalities follow:
\beq\label{fv0_to_fbvk}
\lim_{k\to\infty}\max_{j\in\nint{1}{n_{\epar}}}\|f(\tau,x_0(\tau),u_j^{\epar},v_0(\tau))-f(\tau,x_0(\tau),u_j^{\epar},\bar v_k(\tau))\|=0,\quad\text{for \aee $\tau\in{T}$}.
\eeq

From \fref{fv0_to_fbvk} we obtain the desired relation \fref{bvk_to_qtaue}:
suppose, that for some $\tau\in{T}$ the equality \fref{fv0_to_fbvk} holds, and \fref{bvk_to_qtaue} is not true. Then there exists a sequence $(\bar v_{k_l}(\tau))_{l\in\NA}$ such that
\beq\label{bv_notin_qtau}
\lim_{l\to\infty}\bar v_{k_l}(\tau)=\bar v\notin q_\tau^\epar.
\eeq
Then from the continuity of $f(\cdot)$ and (\ref{fv0_to_fbvk}) we get for all $j\in\nint{1}{n_{\epar}}$ the equality
$$
f(\tau,x_0(\tau),u_j^{\epar},v_0(\tau))=f(\tau,x_0(\tau),u_j^{\epar},\bar v).
$$
The latter imply the relations $\bar v\mysim{\tau x_0(\tau) u^\epar_j} v_0(\tau)$, $ j\in\nint{1}{n_{\epar}}$, that in conjunction are equivalent to the inclusion $\bar v\in q_\tau^\epar$  (see \fref{qe_eq_c_quej}). The last relation contradicts \fref{bv_notin_qtau}. The equality \fref{bvk_to_qtaue} is established.
\end{proof}

\begin{lemma}\label{lem_uniform_disturb_approx_incl}
The strategies $(\Ue)_{\epar>0}$, defined by \emph{\fref{u_0_bar_v_0}--\fref{Ulpi}}, satisfy the equality
\beq\label{XcUe_to_Wq}
\limsup_{\epar\to0}\dist{C({T};\RA^n)}(\Xc(z_0,\Ue),\Wq(z_0))=0, \quad z_0 \in G_0.
\eeq
\end{lemma}

\begin{proof}

1. The equality
\beq\label{yk_to_w}
\lim_{k\rightarrow\infty}\dist{C(T,\RA^n)}(\{y_k(\cdot)\},\Wq(x_{k}(t_0)))=0
\eeq
holds. The proof of \fref{yk_to_w} is based on the definition of $\bar u_k(\cdot)$, on the properties of the sets $\Wq(\cdot)$ (Lemma \ref{lem_z_pr_na_u_stab}) and follows the scheme of the proof of \cite[Theorem 11.3.1]{KraSub88e}.

2. Let estimate the difference $ y_k(\tau)-x_k(\tau)$ for $\tau\in{T}$:
\begin{multline*}%\label{yk_to_xk}
y_k(\tau)-x_k(\tau)=\int\limits_{t_0}^\tau[f(s,y_k(s),\bar u_k(s),\bar v_k(s))-f(s,x_k(s),u_k(s),v_k(s))]\,ds\\
=\int\limits_{t_0}^\tau[f(s,y_k(s),\bar u_k(s),\bar v_k(s))-f(s,x_k(s),\bar u_k(s),\bar v_k(s))]\,ds\\
+\int\limits_{t_0}^\tau[f(s,x_k(s),\bar u_k(s),\bar v_k(s))-f(s,x_k(s),u_k(s),v_k(s))]\,ds.
\end{multline*}
We use the Lipschitz property of right--hand side of the equation \fref{sys} (continuing evaluation):
\begin{multline*}%\label{yk_to_xk}
\le\int\limits_{t_0}^\tau L_f(G)\|y_k(s)-x_k(s)\|ds\\
+\int\limits_{t_0}^\tau\|f(s,x_k(s),\bar u_k(s),\bar v_k(s))-f(s,x_k(s),u_k(s),v_k(s))\|\,ds
\end{multline*}
(here $L_f(G)$ --- the Lipschitz constant of the right--hand side $f(\cdot)$ of the system (\ref{sys}) for second variable in the area $G$). Represent the second integral as sum of two integrals, using the set
$$
M_{\epar}\mydef\bigcup_{i\in\nint{1}{(\nPartk-1)}}[\tau'_{ki},\tau_{ki})
$$
and the identity $u_k(s)=\bar u_k (s)$, $s\in T\setminus M_{\epar}$, $k\in\NA$ (continuing evaluation):
\begin{multline*}%\label{yk_to_xk}
\le\int\limits_{t_0}^\tau L_f(G)\|y_k(s)-x_k(s)\|ds\\
+\int\limits_{[t_0,\tau]\setminus M_{\epar}}\|f(s,x_k(s),u_k(s),\bar v_k(s))-f(s,x_k(s),u_k(s),v_k(s))\|\,ds\\
+\int\limits_{M_{\epar}}\|f(s,x_k(s),\bar u_k(s),\bar v_k(s))-f(s,x_k(s),u_k(s),v_k(s))\|\,ds.
\end{multline*}

We use the continuity of $f(\cdot)$ in the last variable for the second integral and the majorant $\varkappa(G)$ (see \fref{kappa_G}) --- for the third integral (continuing evaluation):
\begin{multline*}%\label{yk_to_xk}
\le\int\limits_{t_0}^\tau L_f(G)\|y_k(s)-x_k(s)\|ds\\
+\int\limits_{[t_0,\tau]\setminus M_{\epar}}\|f(s,x_k(s),u_k(s),\bar v_k(s))-f(s,x_k(s),u_k(s),v_0(s))\|\,ds\\
+\int\limits_{[t_0,\tau]\setminus M_{\epar}}\mfv(\|v_0(s)-v_k(s)\|)\,ds+2\varkappa(G)\mes(M_{\epar}).
\end{multline*}
In the second integral we use the continuity of $f(\cdot)$ in the third variable (continuing evaluation):
\begin{multline*}%\label{yk_to_xk}
\le\int\limits_{t_0}^\tau L_f(G)\|y_k(s)-x_k(s)\|ds\\
+\int\limits_{[t_0,\tau]\setminus M_{\epar}}\|f(s,x_k(s),u_k^\epar(s),\bar v_k(s))-f(s,x_k(s),u_k^\epar(s),v_0(s))\|\,ds\\
+\int\limits_{[t_0,\tau]\setminus M_{\epar}}\mfv(\|v_0(s)-v_k(s)\|)\,ds+2(\vartheta-t_0)\mfu(\epar)+2\varkappa(G)\mes(M_{\epar}),\\
\le\int\limits_{t_0}^\tau L_f(G)\|y_k(s)-x_k(s)\|ds+\int\limits_{[t_0,\tau]\setminus M_{\epar}}\mfv\left(\dist{\RA^q}\left(\{\bar v_k(s)\},q^{u^\epar_k(s)}_s\right)\right)\,ds\\
+\int\limits_{[t_0,\tau]\setminus M_{\epar}}\mfv(\|v_0(s)-v_k(s)\|)\,ds+2(\vartheta-t_0)\mfu(\epar)+2\varkappa(G)\mes(M_{\epar}),
\end{multline*}
where $u_k^\epar(s)\in\argmin_{j\in\nint{1}{n_\epar}} \|u_j^\epar-u_k(s)\|$; note that, by definition of $\epar$--net, the inequality $\|u_k^\epar(s)-u_k(s)\|\le\epar $ holds; $\mfu(\cdot)$ --- modulus of continuity of $f(\cdot)$ to the third argument:
$$%\beqnt%\label{mfu}
\mfu(\delta)\mydef\max_{|u-u'|\le\delta\atop{(\tau,x)\in G\atop v\in\QC, u,u'\in\PC}}\|f(\tau,x,u,v)-f(\tau,x,u',v)\|,\quad\lim_{\delta\rightarrow+0}\mfu(\delta)=0.
$$%\eeq

Using the inequality $\mes(M_{\epar})\le\epar(\vartheta-t_0)$ and the definition of the set $q^\epar_s$ we get 
\beq\label{yk_to_xk_prel}
\|y_k(\tau)-x_k(\tau)\|\le\int\limits_{t_0}^\tau L_f(G)\|y_k(s)-x_k(s)\|ds+\Psi_{1k},
\eeq%nd{multline*}
where
$$
\Psi_{1k}\mydef\int\limits_{T}\big[\mfv\bigl(\dist{\RA^q}(\{\bar v_k(s)\},q_s^\epar)\bigr)+\mfv(\|v_k(s)-v_0(s)\|)\big]\,ds+2(\vartheta-t_0)\big(\mfu(\epar)+\varkappa(G)\epar\big).
$$
We apply to \fref{yk_to_xk_prel} the Gronwall lemma (see \cite[Theorem II.4.4]{Wargae}):
\beq\label{xk_yk_fin_est}
\|y_k(\tau)-x_k(\tau)\|\le\Psi_{1k}\big(1+(\vartheta-t_0)L_f(G)\exp((\vartheta-t_0)L_f(G))\big).
\eeq

4. Lemma \ref{lem_bvk_to_qtaue}, the convergences \fref{vk_to_v0}, \fref{yk_to_w} and the inequality \fref{xk_yk_fin_est} imply the estimate
\begin{multline*}
\dist{C(T,\RA^n)}(\{x_0(\cdot)\},\Wq(z_0))\\
\le2(\vartheta-t_0)\big[1+(\vartheta-t_0)L_f(G)\exp\big((\vartheta-t_0)L_f(G)\big)\big]\big(\mfu(\epar)+\varkappa(G)\epar\big),
\end{multline*}
that, in view of the choice of $ x_0(\cdot)\in\Xc(z_0,\Ue)$, leads to the equality \fref{XcUe_to_Wq}.
\end{proof}

\subsection{The proof of Theorem \ref{teo_Q_EQ_C}}

By  Lemma \ref{lem_uniform_disturb_approx_incl} the equalities \fref{XcUe_to_Wq} hold. These equalities and equality \fref{z_pr_na_gives_opt_guarantee} of Lemma \ref{lem_z_pr_na_u_stab} imply the inequality \fref{Ue_approx_gamma_ps}:
$$
\limsup_{\epar\to0}\Gc(z_0,\Ue)=\limsup_{\epar\to0}\max_{x(\cdot)\in\Xc(z_0,\Ue)}\qndx(x(\cdot))\le\max_{x(\cdot)\in \Wq(z_0)}\qndx(x(\cdot))=\Gq(z_0).
$$

%%%%%%%%%%%%%%%%%%%%%%%%%%%%%%%%%%%%%%%%%%%%%%%%%%%%%%%%%%%
\section*{Acknowledgments}
This work was supported by the Russian Foundation for Basic Research (project no.  12-01-00290),  by the Program for Fundamental Research of Presidium of the Russian Academy of Sciences ``Dynamic Systems and Control Theory'', by the Ural Branch of the Russian Academy of Sciences (project no. 12-$\Pi$-1-1002).

%\theoremstyle{plain}
%\bibliographystyle{unsrt}
%\bibliography{D:/Dropbox/CURRENT/ref/mybib,D:/Dropbox/CURRENT/ref/allbib}

\begin{thebibliography}{10}

\bibitem{KraSub88e}
N.~N. Krasovskii and A.~I. Subbotin.
\newblock {\em Game-theoretical control problems}.
\newblock Springer-Verlag New York, Inc., 1988.

\bibitem{SubChe81e}
A.I. Subbotin and A.G. Chentsov.
\newblock {\em Optimization of Guarantee in Control Problems}.
\newblock Nauka, M., 1981.
\newblock (in Russian).

\bibitem{KrasIZVD70e}
N.~N. Krasovskii.
\newblock {\em Igrovye zadachi o vstreche dvizhenii [Game Problems on the
  motions]}.
\newblock Nauka, Moscow, 1970.
\newblock (in Russian).

\bibitem{Kra71DANe}
N.~N. Krasovskii.
\newblock Programm absorption in differential games.
\newblock {\em Dokl. Acad. Nauk SSSR}, 201(3), 1971.
\newblock (in Russian).

\bibitem{KrasUDS85e}
N.~N. Krasovskii.
\newblock {\em Control of a dynamical system}.
\newblock Nauka, Moscow, 1985.
\newblock (in Russian).

\bibitem{BarSub70e}
N.~N. Barabanova and A.~I. Subbotin.
\newblock On the continuous evasion strategies in pursuit--evasion games.
\newblock {\em Prikl. math. mech.}, 34(5):796--803, 1970.
\newblock (in Russian).

\bibitem{BarSub71e}
N.~N. Barabanova and A.~I. Subbotin.
\newblock On the classes of strategies in the differential games of evasion.
\newblock {\em Prikl. math. mech.}, 35(385--392):385--392, 1971.
\newblock (in Russian).

\bibitem{Kry91}
A.~V. Kryazhimskii.
\newblock {\em The problem of optimization of the ensured result:
  unimprovability of full-memory strategies}, chapter~37, pages 636--675.
\newblock World Scientific, 1991.

\bibitem{Ser_DAN2013et}
D.A. Serkov.
\newblock Optimization of guaranteed results under functional restrictions on
  the dynamic disturbance.
\newblock {\em Doklady Mathematics}, 87(3):310--313, 2013.

\bibitem{KryOsi83e}
A.~V. Kryazhimskii and Yu.~S. Osipov.
\newblock On control modeling in a dynamical system.
\newblock {\em Izv. Akad. Nauk SSSR, Tekhn. Kibern}, 2:51--60, 1983.
\newblock (in Russian).

\bibitem{KryOsi95}
Yu.S. Osipov and A.V. Kryazhimskii.
\newblock {\em Inverse Problems for Ordinary Differential Equations: Dynamical
  Solutions}.
\newblock Gordon and Breach Publishers, London, 1995.

\bibitem{Ser_TRIMM2014e}
D.A Serkov.
\newblock On non-improvability of full--memory strategies in problems of
  optimization of the guaranteed result.
\newblock {\em Trudy Inst. Mat. i Mekh. UrO RAN}, 20(3), 2014.
\newblock (in Russian).

\bibitem{Wargae}
J.~Warga.
\newblock {\em Optimal control of differential and functional equations}.
\newblock Academic Press New York, 1972.

\bibitem{Ser_UDSU_09e}
D.A. Serkov.
\newblock On a property of constructive motions.
\newblock {\em Vestnik Udmurtskogo Universiteta. Matematika. Mekhanika.
  Komp'yuternye Nauki}, 3(3):98--103, 2009.
\newblock (in Russian).

\bibitem{Ser_UDSU_10e}
D.A. Serkov.
\newblock On a property of the constructive motions {II}.
\newblock {\em Vestnik Udmurtskogo Universiteta. Matematika. Mekhanika.
  Komp'yuternye Nauki}, 3(3):64--69, 2010.
\newblock (in Russian).

\bibitem{Oxtobye}
J.C. Oxtoby.
\newblock {\em Measure and category}, volume~2 of {\em Graduate Texts in
  Mathematics}.
\newblock Springer US, New York, 1971.

\bibitem{Natansonet}
I.P. Natanson.
\newblock {\em Theory of Functions of Real Variable}.
\newblock F.Ungar, New York, 1961.

\bibitem{Rudin1986}
Walter Rudin.
\newblock {\em Real and Complex Analysis}.
\newblock International Series in Pure and Applied Mathematics. McGraw-Hill
  Science/Engineering/Math, 1986.

\bibitem{Ser_MGTA2012e}
D.A. Serkov.
\newblock Guaranteed control under functionally restricted disturbances.
\newblock {\em Matematicheskaya Teoriya Igr i Ee Prilozheniya}, 4(2):71--95,
  2012.
\newblock (in Russian).

\end{thebibliography}

\end{document}